\documentclass[12pt,a4paper]{article}
\usepackage[utf8x]{inputenc}
\usepackage{ucs}
\usepackage{amsmath, amsthm}
\usepackage{amsfonts}
\usepackage{amssymb}
\usepackage{graphicx}
\newtheorem{theorem}{Theorem}
\theoremstyle{definition}
\newtheorem{defn}{Definition}
\usepackage[left=2cm,right=2cm,top=2cm,bottom=2cm]{geometry}
\author{S.Gulzar$^{\dagger,*},$ Ravinder Kumar$^{1}$, Mudassir A Bhat$^{1,2}$}

\date{}
\begin{document}
\title{On operators preserving inequalities between polynomials }
\maketitle
\begin{center}
$^{\dagger}$Department of Mathematics, GCET Kashmir, India\\
$^1$Department of Mathematics, Chandigarh University, India\\
$^{1,2}$Department of Mathematics, GDC Anantnag, Kashmir, India
\end{center}
\begin{abstract}
In this review paper, we explore operator aspects in extremal properties of Bernstein-type polynomial inequalities. We shall also see that a linear operator which send polynomials to polynomials and have zero-preserving property naturally preserve Bernstein's inequality.
\end{abstract}
-------------------------------\\
{$^*$ Corresponding author}\\ $^*$email: sgmattoo@gmail.com

\section{Introduction}
Polynomials play an important role in applied sciences in general and are fundamental objects in harmonic and complex analysis in particular. Polynomials find their place everywhere in Mathematics. The extremal properties of polynomials are extensively used in approximation theory and numerical analysis. The problems related to applied mathematics are dependent on polynomials to a greater degree.
The study of extremal properties of the derivatives of polynomials started with some scientific investigations by well-known Russian chemist Dmitri Mendeleev \cite{mmr}. Mathematically, Mendeleev's problem amounts to estimate $|p^\prime(t)|$ on $-1\leq t\leq 1$ for a quadratic polynomial $p(t)$ satisfying $-1\leq p(t)\leq 1$ for $-1\leq t\leq 1.$  Mendeleev was himself able to find the solution. In fact, he found that $-4\leq p^\prime(t)\leq 4$ for $t\in [-1,1]$ and $p(t)=1-2t^2$ is the extemal polynomial. Mendeleev communicated his conclusions to distinguished Russian Mathematician Andrey Andreyevich Markov \cite{markov}, who extended Mendeleev's result to general real polynomials and obtained following generalization.
\begin{theorem}
If a polynomial $p(t)$ of degree $n$ with real coefficients satisfy $-1\leq p(t)\leq 1$ over $[-1,1]$ then
\begin{align*}
|p^\prime(t)|\leq n^2\quad\text{for}\quad -1\leq t\leq 1.
\end{align*}
The inequality cannot be improved as equality holds for $p(t)=\cos(n\arccos t).$ 
 \end{theorem} 
This result of Markov marked the begining in the research of extremal properties of polynomials. Subsequently, his younger brother Vladimir A. Markov \cite{markov2} obtained a best possible estimate  for $s$th derivative, $|p^{(s)}(t)|$ where $s\leq n.$

It was a few years later, another Russian Mathematician Serge Bernstein \cite{BNS} while looking for its analogue for complex polynomials over  $|z|\leq 1$ obtained following theorem. 
\begin{theorem}\label{thm1}
If $f(z)$ is a complex-polynomial having degree at most $n,$ then for $|z|= 1$
\begin{align}\label{e1}
|f^\prime(z)|\leq n ~\underset{|z|=1}{\max}|f(z)|.
\end{align}
The bound is sharp as is shown by $f(z)=az^n,$ $a\neq 0.$
\end{theorem}
\begin{proof}\renewcommand{\qedsymbol}{}
Let $M$ represent the maximum of $|f(z)|$ over $|z|\leq 1.$ Therefore, $|f(z)|<\tau M|z^n|$ for $\tau\in\mathbb{C}$ with $|\tau|>1$ and $|z|\leq 1.$ This implies by invoking Rouch\'{e}'s theorem that the polynomial $F(z)=f(z)+\tau Mz^n$ has no zeros outside $|z|< 1.$

To complete the proof, two important properties of the derivative are needed: (1) Linearity (2) zero preserving property for polynomials. 

The first property is obvious. The second property is a special case of Gauss Luc\'{a}s theorem. According to which for any non-constant polynomial $f(z),$ all the zeros of $ f^\prime (z)$ are completely enclosed in the smallest convex polygon formed by the zeros of $f(z).$  

Therefore by utilizing this zero-preserving property along the linearity of the derivative, we conclude that every zeros of $F^\prime(z)=f^\prime(z)+\tau nMz^{n-1}$ is in $|z|< 1.$ This implies 
\begin{align}\label{be1}
|f^\prime (z)|\leq nM|z|^{n-1} \qquad\text{for}\quad |z|\geq 1.
\end{align} 
If inequality \eqref{be1} were not true, then we could find $\delta\in\mathbb{C}$ with $|\delta|\geq 1$ such that $|f^\prime (\delta)|> nM|\delta|^{n-1}.$ Now, if we choose $\tau=-f^\prime(\delta)/nM\delta^{n-1}$ then $|\tau|>1$ and $F(\delta)=0,$ which is a  contradiction since no zero of $F(z)$ lie outside $|z|<1.$ This verifies inequality \eqref{be1} and consequently the proof is completed.
\end{proof}
In Bernstein's inequality, equality holds iff $f(z)=az^n$ where $a\neq 0.$ That is, if every zero of $f(z)$ is at origin and the inequality gets strict when $f(z)=0$ has a non-zero root. This suggests that if no zero of $f(z)$ is at origin then the upper bound in \eqref{be1} may be improved. This fact was deeply examined by Paul Erd\"{o}s which lead him to conjecture that if $f(z)$ does not vanish in open disk $|z|<1$ then in the upper bound of \eqref{be1}, $n$ can be replaced by $n/2.$ Peter D Lax \cite{lax} was the first one to prove this conjecture. In fact, he proved:
\begin{theorem}\label{t2}
 For any polynomial $f(z)$ of degree at most $n$ such that $f(z)\neq 0$ for $|z|<1$ then 
 \begin{align*}
|f^\prime (z)|\leq\frac{ n}{2} ~\underset{|z|=1}{\max}|f(z)|\quad where \quad |z|=1.
\end{align*}  
The sharpness of this inequality is shown by $f(z)=az^n+b$ where $|a|=|b|\neq 0.$
\end{theorem}
Professor Ralph P. Boas suggested a research problem to find analogue of Theorem \ref{t2} in the case when $f(z)$ does not vanish for $|z|<R$ $R>0.$ The subcase $R\geq1$ of this problem was resolved by M.A. Malik \cite{malik} but the actual problem still remains open. The results of Lax and Malik were improved in several ways. In particular, M.A. Qazi \cite{qazi} obtained some remarkable improvements by involving the coefficients of polynomials in the bound. 

These kind of inequalities are usually called as Markov Bernstein type inequalities. Several authors have penned down monographs on this topic and numerous papers concerning this area have been published.
\section{$B_n$-operator}
We have seen in Bernstein's inequality that linearity and zero-preserving property of the derivative play an essential role in the proof given above. In fact, it was Qazi Ibadur Rahman and Gerhard Schmeisser who in their celebrated monograph, Analytic Theory of Polynomials \cite[p. 538]{rs}, showed that any operator satisfying these two properties satisfy Bernstein-type inequalities. They named such operators as  $B_n$-operator.
\begin{defn}
\textnormal{ Let the space of polynomials over the field of complex numbers of degree at most $n$ be denoted by $\mathcal{P}_n$ and $\mathcal{P}_{n}^{0}\subset \mathcal{P}_n$ be set of those polynomials whose all zeros lie in $|z|\leq 1.$  A linear operator $T:\mathcal{P}_n\to\mathcal{P}_n$ is called $B_n$-operator if $f\in \mathcal{P}_{n}^{0}$  then $T[f]\in \mathcal{P}_{n}^{0}.$ }
\end{defn}
In the same monograph \cite[p. 539]{rs}, the following theorem is also presented for $B_n$-operators.
\begin{theorem}\label{tb1}
Let $f(z)$ be a complex-polynomial and have degree at most $n$. Moreover, if $\varphi_n(z)=z^n$ and $T$ is a $B_n$-operator, then
\begin{align*}
|T[f](z)|\leq |T[\varphi_n](z)| ~\underset{|z|=1}{\max}|f(z)| \qquad \text{for}\quad |z|\geq 1.
\end{align*}
The inequality cannot be sharpened under this hypothesis as one can easily verify that $f(z)=az^n,$ $a\neq 0$ is extremal.
\end{theorem}
It is pertinent to mention here that the above theorem follows in the similar lines as of the proof of Bernstein's inequality given in the introductory section. To comprehend the similitude in proofs, here we also would like to write-down a proof of above theorem.
\begin{proof}[Proof of Theorem \ref{tb1}]\renewcommand{\qedsymbol}{}
Let  us denote the maximum of $|f(z)|$ over $|z|=1$ by $M.$ Then $|f(z)|<\lambda M|z^n|$ for $\lambda\in\mathbb{C}$ with $|\lambda|>1$ and $|z|\leq 1.$ Again by employing Rouch\'{e}'s theorem, we obtain that all the zeros of polynomial $F(z)=f(z)+\lambda Mz^n$ is in $|z|< 1.$

Here, the linearity and zero-preserving property is carried by $B_n$-operator. Therefore for any $B_n$-operator $T,$  every zero of $T[F](z)=T[f]+\lambda MT[\varphi_n](z)$ reside in $|z|< 1,$ where $\varphi_n(z)=z^n.$ This implies 
\begin{align}\label{b'1}
|T[f] (z)|\leq |T[\varphi_n](z)| M \qquad\text{for}\quad |z|\geq 1.
\end{align} 
If inequality \eqref{b'1} were false, then a complex number say $w$ could be found with $|w|\geq 1$ such that $|T[f](w)|> |T[\varphi_n](w)|M.$ Now, if we choose $\lambda=-T[f](w)/T[\varphi_n](w)M$ then $|\lambda|>1$ and $F(w)=0.$ This is straightforward contradiction about the region containing zeros of $F(z).$  Hence, inequality \eqref{b'1} is valid and accordingly the proof gets complete.
\end{proof} 
From above theorem, it is evident that the operators which are linear and have zero-preserving property preserve Bernstein-type inequalities for polynomials. Although the original proof of Bernstein's inequality does not use these two properties intrinsically. But a natural question arises; Does there exists non-linear zero-preserving operators and preserve Bernstein-type inequalities? This question has not been studied extensively.

The next result which is an analogue of Theorem \ref{t2} for $B_n$-operators was also obtained by Rahman \& Schmeisser \cite[p. 539]{rs}.  
\begin{theorem}
Let $f\in\mathcal{P}_n$ and does not vanish for $|z|<1,$ then for any $B_n$-operator $T,$
\begin{align*}
|T[f](z)|\leq\frac{|T[1](z)|+ |T[\varphi_n](z)|}{2} ~\underset{|z|=1}{\max}|f(z)| \qquad \text{for}\quad |z|\geq 1,
\end{align*}
where $\varphi_n(z)=z^n.$ Indeed, this result is also best possible and $f(z)=u z^n+v,$ $|u|=|v|\neq 0,$ is a polynomial for which equality holds in this inequality.
\end{theorem}
\section{Inequalities Preserving Operators}
It is obvious from Gauss Lucas theorem that $T\equiv \frac{d^{(s)}}{dz^{(s)}}$ represents a $B_n$-operator and an analogous result of Bernstein's inequality for $s$th derivative of polynomial $f(z)$ having degree $n$ takes the following form.
\begin{align*}
\left|\frac{d^{(s)}}{dz^{(s)}} f(z)\right|\leq\left|\frac{d^{(s)}}{dz^{(s)}}z^n\right|\underset{|z|=1}{\max}|f(z)| \qquad for\quad  |z|=1 \,\, (s<n).
\end{align*}
It is natural to seek characterization of operators preserving Bernstein-type inequalities. In this direction, we shall refer the readers to some well-known operators which fall under this characterization. The first result, we ought to present here is due to V.K. Jain \cite{jain}. He proved that an operator $T$ on $\mathcal{P}_n$ which sends $f(z)$ to $T[f](z):=zf^\prime (z)+\frac{n\beta}{2}f(z)$ is a $B_n$-operator if $\beta\in \mathbb{C}$ with $|\beta|\leq 1.$ In this case, the corresponding Bernstein-type inequality is;
\begin{align*}
\left|zf^\prime (z)+\frac{n\beta}{2}f(z)\right|\leq n\left|1+\frac{\beta}{2}\right| \underset{|z|=1}{\max}|f(z)| \qquad \text{for}\quad |z|=1.
\end{align*}
The next operator we are going to present here is due to Aziz et al \cite{nr2010}. In this manuscript, they developed a unified method to arrive at various polynomial inequalities including the inequalities of Bernstein and of Jain. They showed that the operator which sends an $n$th degree polynomial $f(z)$ to $f(Rz)-\beta f(rz)+\alpha\left\{\left(\frac{R+1}{r+1}\right)^n-|\beta|\right\}f(rz)$ is a $B_n$-operator, where $R\geq r\geq 1$ and $\alpha,\beta$ belongs to closed unit circle $|z|\leq 1.$ Their analogue of Bernstein's inequality also included results concerning the growth of polynomials.

For given two $n$th degree polynomials $p(z)=\sum_{j=0}^n\binom{n}{j}a_jz^j$ and $g(z)=\sum_{j=0}^n\binom{n}{j}b_jz^j,$  Shur-Szeg\"{o} composition is defined by $p*g(z)=\sum_{j=0}^n\binom{n}{j}a_jb_jz^j.$ By Shur-Szeg\"{o} composition theorem, if all the zeros of $f(z)$ and $g(z)$ are of modulus at most $r$ and $s$ respectively, then the moduli of the zeros of  $f*g$ cannot exceed $rs.$ Now, if we fix a polynomial $h(z)$ having degree $n$ and all zeros in $|z|\leq 1,$ then the operator which sends an $n$th degree polynomial $f(z)$ to $f*h(z)$ is clearly a $B_n$-operator. In this direction, the next result concerning this operator is due to Suhail Gulzar \& N.A. Rather \cite{sr}.
\begin{theorem}
Let $f(z)$ be a polynomial and have degree $n.$ Further, let $h(z)=\sum_{j=0}^nh_jz^j$ be of degree $n$ and have all zeros in $|z|\leq 1,$ then
\begin{align*}
|f*h(z)|\leq |h_n|\underset{|z|=1}{\max}|f(z)| \qquad \text{for}\quad |z|=1.
\end{align*}
The estimate is sharp and $f(z)=az^n$ is the extremal polynomial.
\end{theorem}
 For the choice $h(z)=\sum_{j=0}^n\binom{n}{j}jz^j$ the above inequality condenses to Bernstein's inequality.
 
 Like Shur-Szeg\"{o} composition, several other types of 'compositions' were studied by Mathematicians. Moris Marden in his comprehensive book \cite[p. 86]{mm} also discussed a composition of two polynomials wherein he proved that if the zeros of polynomials $p(z)$ of degree $n$ and  
$g(z)=\mu_0+\binom{n}{1}\mu_1z+\ldots+\binom{n}{n}\mu_nz^n$
respectively are in $|z|\leq r$ and in the circular region: 
$|z|\leq s|z-\sigma|, s>0,$
then the zeros of
\begin{align*}
f(z)=\mu_0p(z)+\mu_1p^{\prime}(z)\dfrac{(\sigma z)}{1!}+\ldots+\mu_n p^{(n)}(z)\dfrac{(\sigma z)^n}{n!}
\end{align*}
are in the circle $|z|\leq  \max (r,rs).$
Recently, Rather et al \cite{ram} showed in the following result that the hypothesis of Marden's theorem can be relaxed with regards to the degrees of $f(z)$ and $g(z).$
\begin{theorem}\label{th1}
If every zero of $n$th degree polynomial $p(z)$ is of modulus at most $r$ and the polynomial
\begin{align*}
g(z)=\mu_0+\binom{n}{1}\mu_1z+\ldots+\binom{n}{m}\mu_mz^m
\end{align*}
has no zero outside the region $|z| \leq s|z-\sigma|$, $s>0$, then the polynomial
\begin{align*}
f(z)=\mu_0p(z)+\mu_1p^{\prime}(z)\dfrac{(\sigma z)}{1!}+\ldots+\mu_m p^{(m)}(z)\dfrac{(\sigma z)^m}{m!}
\end{align*}
has all its zeros in $|z| \leq  \max(r,rs)$
\end{theorem}
By using this Theorem, they \cite{ram} introduced an operator $N$ which takes $f\in\mathcal{P}_n$ into $N[f]\in\mathcal{P}_n$ where 
$N[f](z):=\sum_{i=0}^{m}\mu_i\left(\frac{nz}{2}\right)^i\frac{f^{(i)}(z)}{i!},$ and
 $\mu_i,$ $ i=0,1,2,\ldots,m$ are chosen such that every zeros of
$\phi(z) =  \sum_{i=0}^{m} \binom{n}{i}\mu_i z^i, m \leq n$
lie in $|z| \leq |z - \frac{n}{2}|.$ It is obvious from Theorem \ref{th1} that the operator $N$ is a $B_n$-operator. They also obtained the following Bernstein-type inequality for this operator.
 \begin{theorem}\label{co1}
If $f(z)$ is a polynomial and have degree at most $n,$  then 
\begin{align} \label{x7}
|N[f](z)|\leq |N[\varphi_n](z)| \max_{|z|=1}|f(z)|, \qquad for \quad |z|\geq 1,
\end{align}
where $\varphi_n (z)=z^n$.
This bound is also best possible and for $f(z)=e^{i \alpha} \max_{|z|=1}|f(z)| z^n$, $\alpha \in \mathbb{R},$ equality holds in \eqref{x7}.
\end{theorem}
\section{Bernstein's Inequality in $L_p$-norm}
 Let $f\in \mathcal{P}_n,$ then the Hardy space norm is defined by $$ \| f\|_{p} ={\bigg(\frac{1}{2\pi}\int\limits_{0}^{2\pi}|f(e^{i \theta})|^p d\theta \bigg)}^{1/p}, \quad 0<p< \infty;$$  and the Mahler measure by $$ \| f\|_{0} =\exp{\bigg(\frac{1}{2\pi}\int\limits_{0}^{2\pi}\ln\left(|f(e^{i \theta})|\right) d\theta \bigg)}.$$  
 It is not hard to see that $\lim_{p\to0+}\|f\|_{p}=\|f\|_{0}.$ Also note that the supremum norm satisfies $ \lim_{p\to\infty}\|f\|_{p}=\max_{|z|=1}|f(z)|.$ 
 
  If $f\in\mathcal{P}_n$ then
\begin{align}\label{i1}
\| f^{\prime}\|_{p} \leq n\,\| f\|_{p}, \quad 0\leq p\leq \infty
\end{align}
Inequality \eqref{i1} is an extension of \eqref{e1} in the Hardy space norm settings. Zygmund \cite{ZG} proved this inequality for $p\geq 1.$ It was unknown for quite a long time, whether inequality \eqref{i1} holds for $0<p<1$ or not.  Later, this case was settled by Arestov \cite{ART}. For $p=0,$ \eqref{i1} is a consequence of a remarkable inquality of De Bruijn and Springer \cite{DS}. 

Like the above extension of Bernstein's inequality by Zygmund, most of the Bernstein-type inequalities brought up in the first three sections of this review paper have also been extended in the Hardy space norm. Those inequalities are usually called as Zygmund-type inequalities. Moreover, all the Bernstein-type inequalities in sup-norm mentioned in this paper can be easily derived from Theorem \ref{tb1} by choosing the $B_n$-operator $T$ suitably.  Like sup-norm counterparts, from the expected extension of Theorem \ref{tb1}, the Zymund-type inequalities for different $B_n$-operators should be derivable in a unified manner.

In one of most cited papers of Arestov \cite{ART}, one of his results may be considered as an extension of Theorem \ref{tb1} in Hardy space norm. For $f(z)=\sum_{j=0}^{n}a_jz^j \in\mathcal{P}_n$ and $\vartheta=(\vartheta_0,\vartheta_1,\ldots,\vartheta_n)\in \mathbb{C}^{n+1},$  Arestov introduced the linear operator $\Lambda_{\vartheta}f(z)=\sum_{j=0}^{n}\vartheta_j a_jz^j$ in $\mathcal{P}_n.$  He called $\Lambda_{\vartheta}$ admissible if one of the following properties is preserved by it.
\begin{enumerate}
\item[(i)]  $f(z)$ having every zeros in $\{z\in\mathbb{C}:|z|\leq 1\}$,
\item[(ii)] $f(z)$ having every zeros in $\{z\in\mathbb{C}:|z|\geq 1\}.$
\end{enumerate}
We end this section with the following result of Arestov.
\begin{theorem}\label{l4} Let $\Theta:\mathbb{R}\to\mathbb{R}$ be a convex \& non-decreasing function and $\Psi(x)=\Theta(\log x)$ then for each $f\in\mathcal{P}_n$ and every admissible operator $\Lambda_\vartheta$, $$\int\limits_{0}^{2\pi}\Psi\left(|\Lambda_\vartheta f(e^{i \theta})|\right)d\theta \leq \int\limits_{0}^{2\pi}\Psi\left(\lambda(\vartheta) |f(e^{i \theta})|\right)d\theta,$$ where $\lambda(\vartheta)=\max (|\vartheta_0|,|\vartheta_n|).$
\end{theorem}
One can get \eqref{i1} from above inequality by taking $\Psi(x)=x^p, $ $p\in(0,\infty).$
\section{Concluding Remarks}
From the proof of theorem \ref{tb1}, one gets an impression that it is hard to get the proof without putting the linearity of a $B_n$-operator or its zero-preserving feature in use. There are many operators defined on $\mathcal{P}_n$ which have zero-preserving property like Nagy's generalized derivative (see \cite{ptjs}) but are not linear. A natural question one can ask here is that, whether theorem \ref{tb1} holds for those operators which are not linear but does preserve location of zeros. Similarly, to what extent one can withhold the zero-preserving property and keep linearity to achieve theorem \ref{tb1}. Moreover, when it comes to the question of extending theorem \ref{tb1} in $L_p$-norm setting via theorem \ref{l4}, a representation of a $B_n$-operator is needed.   

These and many other questions concerning Markov Bernstein type inequalities are intended to be taken up in a subsequent work in the future.

\noindent \textbf{Conflict of Interest:} The authors declare that there are no conflicts of interest regarding the publication of this paper.


\begin{thebibliography}{9}
\bibitem{ART} V. V. Arestov, On integral inequalities for trigonometric polynomials and their derivatives, \emph{Izv. Akad. Nauk SSSR Ser. Mat.,} 45(1981) 3-22 (in Russian). English translatioin; Math. USSR-Izv. 18(1982) 1-17.


\bibitem{nr2010} A. Aziz \& N.A. Rather, Some new generalizations of Zygmund-type inequalities for polynomials, \emph{Math. Ineq. Appl.,} 15 (2) (2012), 469-486.

\bibitem{BNS} S. N. Bernstein, Sur. \'{l}ordre de la meilleure appromation des functions continues par des Polynomes de degree donne, \emph{Mem. Acad. Roy. Belgique}, 12(4), (1912) 1-103. 


\bibitem{DS} N.G. de Bruijin and T.A. Springer, On the zeros of composite polynomials, \emph{Indag. Math.}, 9 (1947), 406-414.

\bibitem{ptjs} P.L. Cheung, T.W. Ng, J. Tsai and S.C.P. Yam, Higher-order, polar and Sz.-Nagy's generalized derivatives of random polynomials with independent and identically distributed zeros on the unit circle, Comput. Methods. Funct. Theory, 15(1) (2015) 159-186.


\bibitem{jain} V.K. Jain, Generalization of certain well known inequalities for polynomials, \emph{Glas. Mat.,} 32 (52) (1997) 45-51.


\bibitem{lax} P.D. Lax, Proof of a conjecture of P.  Erd\"{o}s on the derivative of a polynomial, \emph{Bull. Amer. Math. Soc.,} 50 (1944) 509-513.

\bibitem{malik} M.A. Malik, On the derivatives of a polynomial, J. London Math. Soc., 1 (1969), 57-60.

\bibitem{mm} M. Marden, \emph{Geometry of polynomials}, Math Surveys, No. 3. Amer. Math. Soc. Providence (1949).

\bibitem{markov} A.A. Markov, On a problem of D.I. Mendeleev (Russian), \emph{Zapishi Imp. Akad.Nauk}, I12 (1889) 1-24.

\bibitem{markov2} V.A.Markov, \"{U}ber Polynome die in einem gegebenen Intervalle m\"{o}glichst wenig von null abweichen, \emph{Math. Annalen.}, 77 (1916), 213-258.

\bibitem{mmr} {G. V.  Milovanovi\'{c}, D. S. Mitrinovi\'{c} and Th. M. Rassias}, \emph{Topics In Polynomials: Extremal Problems, Inequalities, Zeros}, World Scientific Publications (1994).

\bibitem{ps} G. P\'{o}lya and G. Szeg\"{o}, \emph{Problems and Theorems in Analysis II}, Springer-Verlag (1976).

\bibitem{qazi} M.A. Qazi, On the maximum modulus of polynomials, Proc. Amer. Math. Soc., 115 (1992), 337-343.

\bibitem{rs} Q. I. Rahman and G. Schmeisser, \emph{Analytic theory of Polynomials}, Clarendon Press Oxford (2002).

\bibitem{ram} N.A. Rather, Ishfaq Dar \& Suhail Gulzar, On the zeros of certain composite polynomials and an operator preserving inequalities, Ramanujan J, https://doi.org/10.1007/s11139-020-00261-2 (2020).

\bibitem{sr} Suhail Gulzar \& N.A. Rather, On a composition preserving inequalities between polynomials, \emph{J. Contemp. Math. Anal.,} 53 (1) (2018) 21-26.

\bibitem{ZG} A. Zygmund, A remark on conjugate series, \emph{Proc. Lond. Math. Soc.} 34(1932) 292-400.

\end{thebibliography}
\end{document}